\documentclass[12pt,a4paper]{amsart}

\usepackage[utf8]{inputenc}
\usepackage[T1]{fontenc}
\usepackage[english]{babel}
\usepackage{cite}

\usepackage{indentfirst}
\usepackage{amssymb}
\usepackage{amsfonts}
\usepackage{amsmath}
\usepackage{amsthm}
\usepackage[top=2.5cm, bottom=2.5cm, left=2.5cm, right=2.5cm]{geometry}
\usepackage{amsopn}
\usepackage[colorlinks]{hyperref}
\usepackage{mathrsfs}
\usepackage{amsxtra}
\usepackage{color}
\usepackage{dsfont}

\newcommand{\al}{\alpha}
\newcommand{\be}{\beta}
\newcommand{\si}{\sigma}
\newcommand{\om}{\omega}

\theoremstyle{plain}
\newtheorem{thm}{Theorem}

\newtheorem{lem}[thm]{Lemma}

\theoremstyle{definition}

\newtheorem*{example*}{Example}

\newtheorem*{rem*}{Remark}

\newcommand{\de}{\delta}
\newcommand{\Om}{\Omega}
\newcommand{\ga}{\gamma}
\newcommand{\Sd}{\mathbb{S}^{d-1}}

\newcommand{\R}{\mathbb{R}}

\DeclareMathOperator{\2F1}{_2F_1}

\DeclareSymbolFont{bbsymbol}{U}{bbold}{m}{n}
\DeclareMathSymbol{\ind}{\mathbin}{bbsymbol}{'061}

\title[Best constants for Hardy inequalities in Triebel--Lizorkin spaces]{Best constants for Hardy inequalities in Triebel--Lizorkin spaces}
\author[M{.} Kijaczko]{Micha\l{} Kijaczko}


\address[ M.K.]{Faculty of Pure and Applied Mathematics\\ Wroc{\l}aw University 
	of Science and Technology\\
	Wybrze\.ze Wyspia\'nskiego 27,
	50-370 Wroc{\l}aw, Poland
}
\email{michal.kijaczko@pwr.edu.pl}

\begin{document}
\maketitle
\begin{abstract}
We find sharp constants in fractional Hardy inequalities for weighted Triebel--Lizorkin seminorms on the whole space and half-spaces. Our results generalize recently obtained weighted fractional Hardy inequalities for Gagliardo seminorms, but are new even for the unweighted case.    
\end{abstract}
\section{Introduction}
Hardy inequalities play a key role in such areas of research as mathematical analysis, partial differential equations, and stochastic processes. Among them, in recent years, fractional Hardy inequalities have attracted significant attention, due to their connection with important branches like nonlocal problems or theory of diffusions, where they arise from studies on Dirichlet forms. These inequalities are strictly related to the structure of \emph{fractional Sobolev spaces}. Recall that, for $0<s<1$, $p\geq 1$ and $\Omega\subset\R^d$ being an open set, this space is defined as
\begin{equation}\label{Wsp}
W^{s,p}(\Om)=\left\{u\in L^p(\Om)\colon\int_{\Om}\int_{\Om}\frac{|u(x)-u(y)|^p}{|x-y|^{d+sp}}\,dy\,dx<\infty\right\}.
\end{equation}
The space $W^{s,p}(\Om)$ is a Banach space endowed with the norm 
$$
\|f\|_{W^{s,p}(\Om)}=\left(\|f\|_{L^p(\Om)}+\int_{\Om}\int_{\Om}\frac{|u(x)-u(y)|^p}{|x-y|^{d+sp}}\,dy\,dx\right)^{\frac{1}{p}}.
$$
In this article, however, we deal with a variant of \eqref{Wsp}, called \emph{Triebel--Lizorkin space}. For $d\geq 1$, $\max\left\{0,\tfrac{d(p-q)}{pq}\right\}<s<1$, $p,q\geq 1$, the latter may be defined as 
$$
F^{s}_{p,q}(\Om)=\left\{f\in L^{\max\{p,q\}}(\Om)\colon \|f\|_{L^q(\Om)}+\int_{\Om}\left(\int_{\Om}\frac{|u(x)-u(y)|^p}{|x-y|^{d+sp}}\,dy\right)^{\frac{q}{p}}dx<\infty\right\},
$$
see \cite[Theorem 1.2, Theorem 1.3]{MR3667439}. The double integral appearing on the right-hand side will be refered to as \emph{Triebel--Lizorkin seminorm}. In recent years, Triebel--Lizorkin spaces have attracted considerable attention. For more information about (weighted) fractional Sobolev and Triebel--Lizorkin spaces, as well as fractional Hardy inequalities, we refer to \cite{MR4800921,MR2944369,MR4708667,MR4705882, MR4675974,MR4686398,MR3667439,MR4932272, MR4961027,MR4849884, MR4720167, MR1731740} and references therein. Noteworthy, Hardy inequalities in Triebel--Lizorkin spaces on general domains were studied in details by Ihnatsyeva and V\"ah\"akangas in \cite{MR3205532, MR3322431}.
\subsection{Hardy inequality with power weights on $\R^d$}
Our first result is the fractional Hardy inequality for weighted Triebel--Lizorkin seminorm on $\R^d$.
\begin{thm}\label{thm1}
Let $0<s<1$, $p\geq q$ and $d\geq 1$. Assume that the parameters $\al,\be$ satisfy
\begin{align}
\label{assumption1}\al<sq-(1-\tfrac{q}{p})d, &\text{\quad} \be<sp,\quad \al,\be>-d,\quad\al+\tfrac{q\be}{p}>-d.
\end{align}

Let $u\in C_c^1(\R^d)$ if $sq-\al-\tfrac{q\be}{p}<d$ and $u\in C_c^1(\R^d\setminus\{0\})$ if $sq-\al-\tfrac{q\be}{p}>d$. Then, the following Hardy inequality holds,
\begin{equation}\label{HardyTriebelLizorkin}
\int_{\R^d}\left(\int_{\R^d}\frac{|u(x)-u(y)|^p}{|x-y|^{d+sp}}|y|^{\be}\,dy\right)^{\frac{q}{p}}|x|^{\al}\,dx\geq \mathcal{C}^{p,q}_{d,s,\al,\be}\int_{\R^d}\frac{|u(x)|^q}{|x|^{sq-\al-\frac{q\be}{p}}}\,dx.
\end{equation}
The constant $\mathcal{C}^{p,q}_{d,s,\al,\be}$ is sharp and equals
\begin{equation}\label{CTriebelLizorkin}
\mathcal{C}^{p,q}_{d,s,\al,\be}=\left(\int_{0}^{1}t^{sp-1}\left(t^{-\left(\frac{p}{q}-1\right)d-\frac{p\al}{q}}+t^{-\be}\right)\left|1-t^{\left(d+\al+\frac{q\be}{p}-sq\right)/q}\right|^p\Phi_{d,s,p}(t)\,dt\right)^{\frac{q}{p}},
\end{equation}
where
\begin{equation}\label{definitionofPhi}
\Phi_{d,s,p}(t)=\begin{cases}
\om_{d-2}\displaystyle\int_{-1}^{1}\frac{\left(1-r^{2}\right)^{\frac{d-3}{2}}}{\left(1-2tr+t^2
  \right)^{\frac{d+sp}{2}}}\,dr,\,d\geq 2 \\
(1-t)^{-1-sp}+(1+t)^{-1-sp},\,d=1.
\end{cases}
\end{equation}
\end{thm}

Here and elsewhere, $\Sd$ is the unit sphere in $\R^d$ and $\om_{d-1}=\left|\mathbb{S}^{d-1}\right|=\frac{2\pi^{d/2}}{\Gamma\left(d/2\right)}$. It may be checked that the assumptions in \eqref{assumption1} enforce that the constant $\mathcal{C}^{p,q}_{d,s,\al,\be}$ and weighted Triebel--Lizorkin seminorm are finite, see Lemma \ref{Lemma}. 

In the unweighted case $\al=\be=0$ and $p=q=2$, the inequality \eqref{HardyTriebelLizorkin} was given by Herbst \cite{MR436854}, Beckner \cite{MR1254832} and Yafaev \cite{MR1717839}. For general $p=q\geq 1$, \eqref{HardyTriebelLizorkin} with generic constant was obtained by Maz'ya and Shaposhnikova \cite{MR1940355} and the sharp form of the constant was found by Frank and Seiringer in their seminal paper \cite{MR2469027}. In the weighted setting, \eqref{HardyTriebelLizorkin} was given in \cite{MR3626031} by  Abdellaoui and Bentifour for $sp<d$ and $\al=\be\in(\tfrac{sp-d}{2},0]$ and in \cite{MR4705882} by Dyda and Kijaczko for $\al,\be,\al+\be\in(-d,sp)$

For $p\neq q$, according to our best knowledge, no sharp results were previously known, even in the unweighted case. We remark here that inequality similar to \eqref{HardyTriebelLizorkin} on homogeneous Lie groups with more general weights and $1<p\leq q$, but without addressing the sharp constant, was recently obtained in \cite[Theorem 2.7]{MR4804831}. In this context, we also refer to \cite{MR1299571}.

\subsection{Hardy inequality on a half-space}
Our second result is the Hardy inequality for Triebel--Lizorkin seminorm on the half-space $\R^d_+=\{x\in\R^d\colon x_d>0\}$.
\begin{thm}\label{thm2}
Let $0<s<1$, $d\geq 1$, $p\geq q$. Assume that the parameters $\al,\be$ satisfy
\begin{align}
\label{assumption2halfspace}\al<sq-(1-\tfrac{q}{p}), &\text{\quad} \be<sp,\quad \al,\be>-1,\quad\al+\tfrac{q\be}{p}>-1.
\end{align}

Let $u\in C_c^1\left(\overline{\R^d_+}\right)$ if $sq-\al-\tfrac{q\be}{p}<1$ and $u\in C_c^1(\R^d_+)$ if $sq-\al-\tfrac{q\be}{p}>1$. Then, the following Hardy inequality holds,
\begin{equation}\label{half-space}
 \int_{\R^d_+}\left(\int_{\R^d_+}\frac{|u(x)-u(y)|^p}{|x-y|^{d+sp}}y_d^{\be}\,dy\right)^{\frac{q}{p}}x_d^{\al}\,dx\geq \mathcal{D}^{p,q}_{d,s,\al,\be}\int_{\R^d}\frac{|u(x)|^q}{x_d^{sq-\al-\frac{q\be}{p}}}\,dx.
\end{equation}
The constant $\mathcal{D}^{p,q}_{d,s,\al,\be}$ is sharp and equals
\begin{equation}\label{DTriebelLizorkin}
\mathcal{D}^{p,q}_{d,s,\al,\be}=\left(\frac{\pi^{\frac{d-1}{2}}\Gamma\left(\frac{1+sp}{2}\right)}{\Gamma\left(\frac{d+sp}{2}\right)}\int_{0}^{1}\frac{t^{sp-1}\left(t^{-\be}+t^{-\frac{p\al}{q}-\frac{p}{q}+1}\right)\left|1-t^{(sq-\al-\frac{q\be}{p}-1)/q}\right|^p}{(1-t)^{1+sp}}\,dt\right)^{\frac{q}{p}}.
\end{equation}
\end{thm}
Again, by Lemma \ref{Lemma} below, \eqref{assumption2halfspace} implies that the constant $\mathcal{D}^{p,q}_{d,s,\al,\be}$ and weighted Triebel--Lizorkin seminorm are finite. When $p=q$, the sharp constant in \eqref{half-space} in the unweighted case $\al=\be=0$ was found  by Bogdan and Dyda in \cite{MR2663757} for $p=2$ and Frank and Seiringer in \cite{MR2723817} for general $p\geq 1$. The weighted result was established by Dyda and Kijaczko in \cite{MR4705882}.

We point that, for $p=q$, fractional Hardy inequalities with sharp constants are derived by means of the ground-state representation, see \cite{MR2469027}. However, for $p\neq q$, no such tools are available. Therefore, our proofs of Theorems \ref{thm1} and \ref{thm2} are completely new even for the case $p=q$ and do not require ground-state representation. 

Unfortunately, our proofs of Theorems \ref{thm1} and \ref{thm2} do not work for $p<q$. Therefore, it is an open problem if the constants \eqref{CTriebelLizorkin} and \eqref{DTriebelLizorkin} are sharp also in the case of $p<q$.\\
\\
\noindent\textbf{Notation}. Through the paper, we will widely use the comparison $a\lesssim b$ or $a\approx b$ to mark that $a\leq cb$ or $c_1b\leq a\leq c_2b$, respectively. The constants appearing in such inequalities will be unimportant and independent on functions or varying parameters.
\section{Proofs}
We start with a technical lemma, which allows us to state under which conditions on parameters the weighted Triebel--Lizorkin seminorms, we are dealing with, are finite. The proof is inspired by \cite[Proof of Lemma 2.1]{MR3420496}, or \cite[Proof of Lemma 2.1]{MR4849884}.
\begin{lem}\label{Lemma}
Let $1\leq p,q<\infty$, $0<s<1$. The following ensure that weighted Triebel--Lizorkin seminorms are finite.\\
\\
(1) Suppose that $\al$ and $\beta$ satisfy \eqref{assumption1}. Then, for $u\in C_c^1\left(\R^d\right)$,
$$
\int_{\R^d}\left(\int_{\R^d}\frac{|u(x)-u(y)|^p}{|x-y|^{d+sp}}|y|^{\be}\,dy\right)^{\frac{q}{p}}|x|^{\al}\,dx<\infty.
$$
(2) Suppose that $\al$ and $\beta$ satisfy \eqref{assumption2halfspace}. Then, for $u\in C_c^1\left(\R^d_+\right)$,
$$
\int_{\R^d_+}\left(\int_{\R^d_+}\frac{|u(x)-u(y)|^p}{|x-y|^{d+sp}}y_d^{\be}\,dy\right)^{\frac{q}{p}}x_d^{\al}\,dx<\infty.
$$
\end{lem}
\begin{proof}
We will only prove the first part of the Lemma, since the second follows the same pattern. Let $u\in C_c^1\left(\R^d\right)$ and let $\text{supp}\,u\subset B_R$, where $B_R=\{x\in\R^d\colon |x|\leq R\}$. We have
\begin{align*}
\int_{\R^d}&\left(\int_{\R^d}\frac{|u(x)-u(y)|^p}{|x-y|^{d+sp}}|y|^{\be}\,dy\right)^{\frac{q}{p}}|x|^{\al}\,dx\\
&=\int_{B_{2R}}\left(\int_{\R^d}\frac{|u(x)-u(y)|^p}{|x-y|^{d+sp}}|y|^{\be}\,dy\right)^{\frac{q}{p}}|x|^{\al}\,dx\\
&+\int_{B_{2R}^c}\left(\int_{B_R}\frac{|u(y)|^p}{|x-y|^{d+sp}}|y|^{\be}\,dy\right)^{\frac{q}{p}}|x|^{\al}\,dx\\
&=:I_1+I_2.
\end{align*}
It holds
\begin{align*}
 I_1&\leq\int_{B_{2R}}\left(\int_{B_{2R}}\frac{|u(x)-u(y)|^p}{|x-y|^{d+sp}}|y|^{\be}\,dy\right)^{\frac{q}{p}}|x|^{\al}\,dx\\
 &\quad+\int_{B_{R}}\left(\int_{B_{2R}^c}\frac{|u(x)|^p}{|x-y|^{d+sp}}|y|^{\be}\,dy\right)^{\frac{q}{p}}|x|^{\al}\,dx\\
 &=:J_1+J_2.
\end{align*}
Let us focus on $J_1$. If $\be$ is nonnegative, we use the bound $|y|^{\be}\leq(2R)^{\be}$ and the proof is easier, so we omit this part. Hence, assume that $\be<0$. Since $|u(x)-u(y)|\lesssim|x-y|$,
\begin{align*}
J_1&\lesssim\int_{B_{2R}}\left(\int_{B_{2R}}\frac{|x-y|^p}{|x-y|^{d+sp}}|y|^{\be}\,dy\right)^{\frac{q}{p}}|x|^{\al}\,dx\\
&\leq\int_{B_{2R}}\left(\int_{B(x,4R)}\frac{|y|^{\be}}{|x-y|^{d-(1-s)p}}\,dy\right)^{\frac{q}{p}}|x|^{\al}\,dx.
\end{align*}
For $u\in L_{loc}^1(\R^d)$, let 
$$
Mu(x)=\sup_{r>0}r^{-d}\int_{B(x,r)}|u(y)|\,dy
$$
be the Hardy--Littlewood maximal operator. It is well known that, if $w$ belongs to the Muckenhoupt class $A_1$, then $Mw(x)\lesssim w(x)$. By \cite[Theorem 1.1]{MR3900847}, the function $w(y)=|y|^{\be}\in A_1$ if and only if $-d<\be\leq 0$. Hence, by the first lemma in \cite{MR312232},
\begin{align*}
    \int_{B(x,4R)}\frac{|y|^{\be}}{|x-y|^{d-(1-s)p}}\,dy\lesssim R^{(1-s)p}|x|^{\be}.
\end{align*}
Consequently, the above estimate yields
$$
J_1\lesssim\int_{B_{2R}}|x|^{\al+\frac{q\be}{p}}\,dx<\infty,
$$
provided that $\al+\frac{q\be}{p}>-d$. Notice that if $\be$ is nonnegative, we obtain the bound $\al>-d$.

We now estimate the integral $J_2$, proceeding as follows. We have $|y|\geq 2R\geq2|x|$, which gives $|y-x|\geq|y|-|x|\geq\tfrac{|y|}{2}$. Therefore,
\begin{align*}
J_2\lesssim \|u\|_{\infty}^{q}\int_{B_{R}}\left(\int_{B_{2R}^c}\frac{dy}{|y|^{d+sp-\be}}\right)^{\frac{q}{p}}|x|^{\al}\,dx<\infty,
\end{align*}
provided that $\be<sp$. 

We now go back to the integral $I_2$. If $|x|\geq 2R$ and $|y|\leq R$, we have $|y-x|\geq |x|-|y|\geq \frac{1}{2}|x|$. It than holds
\begin{align*}
I_2&\lesssim\int_{B_{2R}^c}\left(\int_{B_R}\frac{|y|^{\be}}{|x-y|^{d+sp}}\,dy\right)^{\frac{q}{p}}|x|^{\al}\,dx\\
&\lesssim\int_{B_{2R}^c}\frac{dx}{|x|^{\frac{qd}{p}+sq-\al}}\left(\int_{B_{R}}|y|^{\be}\,dy\right)^{\frac{q}{p}}<\infty,
\end{align*}
if $\al<sq-\left(1-\tfrac{q}{p}\right)d$ and $\be>-d$. Therefore, the proof is complete.
\end{proof}

\begin{proof}[Proof of Theorem \ref{thm1}]
Using polar coordinates, we have
\begin{align*}
\int_{\R^d}&\left(\int_{\R^d}\frac{|u(x)-u(y)|^p}{|x-y|^{d+sp}}|y|^{\be}\,dy\right)^{\frac{q}{p}}|x|^{\al}\,dx\\
&=\int_{0}^{\infty}\int_{\Sd}r^{d-1+\al}\,d\omega\,dr\left(\int_{0}^{\infty}\int_{\Sd}\frac{\left|u(r\om)-u(t\si)\right|^p}{\left|r\om-t\si\right|^{d+sp}}t^{d-1+\be}\,d\si\,dt\right)^{\frac{q}{p}}\\
&=\int_{0}^{\infty}\int_{\Sd}r^{d-1+\al+q\be/p-sq}\,d\om\,dr\left(\int_{0}^{\infty}\int_{\Sd}\frac{\left|u(r\om)-u(rt\si)\right|^p}{\left|\om-t\si\right|^{d+sp}}t^{d-1+\be}\,d\si\,dt\right)^{\frac{q}{p}}.
\end{align*}
where the last equality above follows by substitution $t\mapsto rt$ in the inner integral. By Hölder's inequality, we have
\begin{equation}\label{holder}
\int_{\Sd}\frac{\left|u(r\om)-u(rt\si)\right|^p}{\left|\om-t\si\right|^{d+sp}}\,d\si\geq\left(\int_{\Sd}\frac{\left|u(r\om)-u(rt\si)\right|^q}{\left|\om-t\si\right|^{d+sp}}\,d\si\right)^{\frac{p}{q}}\left(\int_{\Sd}\frac{d\si}{\left|\om-t\si\right|^{d+sp}}\,d\si\right)^{1-\frac{p}{q}}.
\end{equation}
By rotation invariance, the last integral is independent on $\om$ and it holds
\begin{align*}
 \int_{\Sd}\frac{d\si}{\left|\om-t\si\right|^{d+sp}}&=\int_{\Sd}\frac{d\si}{\left|e_d-t\si\right|^{d+sp}}=\int_{\Sd}\frac{d\si}{\left(1-2t\si_d+t^2\right)^{\frac{d+sp}{2}}}\\
 &=\om_{d-2}\int_{-1}^{1}\frac{\left(1-u^2\right)^{\frac{d-3}{2}}}{\left(1-2tu+t^2\right)^{\frac{d+sp}{2}}}\,du=\Phi_{d,s,p}(t)=:\Phi(t),
\end{align*}
where $\Phi$ is given by \eqref{definitionofPhi} and $e_d=(0,0,\dots,0,1)$. 

Denote $\ga=d+\al+\frac{q\be}{p}-sq$. Therefore, applying \eqref{holder},
\begin{align*}
\int_{0}^{\infty}&\int_{\Sd}r^{\ga-1}\,d\om\,dr\left(\int_{0}^{\infty}\int_{\Sd}\frac{\left|u(r\om)-u(rt\si)\right|^p}{\left|\om-t\si\right|^{d+sp}}t^{d-1+\be}\,d\si\,dt\right)^{\frac{q}{p}}\\
&\geq\int_{0}^{\infty}\int_{\Sd}r^{\ga-1}\,d\om\,dr\left(\int_{0}^{\infty}\left(\int_{\Sd}\frac{\left|u(r\om)-u(rt\si)\right|^q}{\left|\om-t\si\right|^{d+sp}}\,d\si\right)^{\frac{p}{q}}t^{d-1+\be}\Phi(t)^{1-\frac{p}{q}}\,dt\right)^{\frac{q}{p}}.
\end{align*}
Recall the Minkowski's integral inequality
$$
\left(\int_{X}\left|\int_{Y}F(x,y)\,dy\right|^{\eta}dx\right)^{1/\eta}\leq\int_{Y}\left(\int_{X}|F(x,y)|^{\eta}\,dx\right)^{1/\eta}dy.
$$
We apply it for $\eta=\frac{p}{q}$, $X=(0,\infty)$, $dx=\Phi(t)^{1-\frac{p}{q}}t^{d-1+\be}\,dt$, $Y=(0,\infty)\times\Sd$, $dy=r^{\ga-1}\,d\om\,dr$ and $F(r,t,\om)=\displaystyle\int_{\Sd}\frac{\left|u(r\om)-u(rt\si)\right|^q}{\left|\om-t\si\right|^{d+sp}}\,d\si$. Hence,
\begin{align}\label{jjjj}
\nonumber\int_{0}^{\infty}&\int_{\Sd}r^{\ga-1}\,d\om\,dr\left(\int_{0}^{\infty}\left(\int_{\Sd}\frac{\left|u(r\om)-u(rt\si)\right|^q}{\left|\om-t\si\right|^{d+sp}}\,d\si\right)^{\frac{p}{q}}t^{d-1+\be}\Phi(t)^{1-\frac{p}{q}}\,dt\right)^{\frac{q}{p}}\\   
&\geq\left(\int_{0}^{\infty}\Phi(t)^{1-\frac{p}{q}}t^{d-1+\be}\,dt\left(\int_{0}^{\infty}\int_{\Sd}\int_{\Sd}r^{\ga-1}\frac{\left|u(r\om)-u(rt\si)\right|^q}{\left|\om-t\si\right|^{d+sp}}\,d\si\,d\om\,dr\right)^{\frac{p}{q}}\right)^{\frac{q}{p}}.
\end{align}    
Observe that, since $|\om-t\si|=|\si-t\om|$, we have
\begin{align*}
\int_{0}^{\infty}\int_{\Sd}\int_{\Sd}r^{\ga-1}\frac{\left|u(r\om)\right|^q}{\left|\om-t\si\right|^{d+sp}}\,d\si\,d\om\,dr&=\Phi(t)\int_{\Sd}\int_{0}^{\infty}r^{\ga-1}|u(r\om)|^q\,dr\,d\om\\
&=\Phi(t)\int_{\R^d}\frac{|u(x)|^q}{|x|^{sq-\al-q\be/p}}\,dx.  
\end{align*}
Using Minkowski's inequality we can estimate
\begin{align}\label{rrrrr}
\nonumber\left(\Phi(t)\int_{\R^d}\frac{|u(x)|^q}{|x|^{sq-\al-q\be/p}}\,dx\right)^{1/q}&=\left(\int_{0}^{\infty}\int_{\Sd}\int_{\Sd}r^{\ga-1}\frac{\left|u(r\om)\right|^q}{\left|\om-t\si\right|^{d+sp}}\,d\si\,d\om\,dr\right)^{1/q}\\
\nonumber&\leq\left(\int_{0}^{\infty}\int_{\Sd}\int_{\Sd}r^{\ga-1}\frac{\left|u(r\om)-u(rt\si)\right|^q}{\left|\om-t\si\right|^{d+sp}}\,d\si\,d\om\,dr\right)^{1/q}\\
\nonumber&\quad+\left(\int_{0}^{\infty}\int_{\Sd}\int_{\Sd}r^{\ga-1}\frac{\left|u(rt\si)\right|^q}{\left|\om-t\si\right|^{d+sp}}\,d\si\,d\om\,dr\right)^{1/q}\\
\nonumber&=\left(\int_{0}^{\infty}\int_{\Sd}\int_{\Sd}r^{\ga-1}\frac{\left|u(r\om)-u(rt\si)\right|^q}{\left|\om-t\si\right|^{d+sp}}\,d\si\,d\om\,dr\right)^{1/q}\\
&\quad+t^{-\ga/q}\left(\Phi(t)\int_{\R^d}\frac{|u(x)|^q}{|x|^{sq-\al-q\be/q}}\,dx\right)^{1/q}.
\end{align}
Similarly,
\begin{align}\label{qqqq}
\nonumber\left(\Phi(t)\int_{\R^d}\frac{|u(x)|^q}{|x|^{sq-\al-q\be/p}}\,dx\right)^{1/q}&=t^{\ga/q}\left(\int_{0}^{\infty}\int_{\Sd}\int_{\Sd}r^{\ga-1}\frac{\left|u(rt\si)\right|^q}{\left|\om-t\si\right|^{d+sp}}\,d\si\,d\om\,dr\right)^{1/q}\\ 
\nonumber&\quad\leq t^{\ga/q}\left(\int_{0}^{\infty}\int_{\Sd}\int_{\Sd}r^{\ga-1}\frac{\left|u(rt\si)-u(r\om)\right|^q}{\left|\om-t\si\right|^{d+sp}}\,d\si\,d\om\,dr\right)^{1/q}\\
&\quad+t^{\ga/q}\left(\Phi(t)\int_{\R^d}\frac{|u(x)|^q}{|x|^{sq-\al-q\be/p}}\,dx\right)^{1/q}.
\end{align}
Thereofore, combining \eqref{rrrrr} with \eqref{qqqq}, we obtain 
$$
\int_{0}^{\infty}\int_{\Sd}\int_{\Sd}r^{\ga-1}\frac{\left|u(r\om)-u(rt\si)\right|^q}{\left|\om-t\si\right|^{d+sp}}\,d\si\,d\om\,dr\geq\left|1-t^{-\ga/q}\right|^q\Phi(t)\int_{\R^d}\frac{|u(x)|^q}{|x|^{sq-\al-q\be/p}}\,dx.
$$
Thus, \eqref{jjjj} is bigger or equal to 
$$
\int_{\R^d}\frac{|u(x)|^q}{|x|^{sq-\al-q\be/p}}\,dx \left(\int_{0}^{\infty}t^{d-1+\be}\left|1-t^{-\ga/q}\right|^p\Phi(t)\,dt\right)^{\frac{q}{p}}.
$$
We now divide the integral over $t$ into integrals over $(0,1)$ and $(1,\infty)$, substitute $t\mapsto\tfrac{1}{t}$ in the second integral and rearrange. Consequently, \eqref{HardyTriebelLizorkin} follows. 

It remains to verify the optimality of the constant \eqref{CTriebelLizorkin}. To this end, let 
$\de=\frac{d+\al+\frac{q\be}{p}-sq}{q}=\frac{\ga}{q}
$
and suppose first that $\de>0$. We want to construct a family of functions which approximate $r\mapsto r^{-\de}$, for $r>0$. To this end, for $n\in\mathbb{N}$, $n\geq 2$, we define the functions 
\begin{equation}\label{u_n}
u_n(r)=\begin{cases}
        n^{\de}-n^{-\de} & \text{ for }0\leq r\leq\tfrac{1}{n}, \\ 
        r^{-\de}-n^{-\de} & \text{ for }\tfrac{1}{n}<r\leq n\, \\
        0 &\text{ for }r>n.
    \end{cases}
\end{equation}
We also denote $u_n(x)=u_n(|x|)$ for $x\in\R^d$. We have
\begin{align}\label{log}
\nonumber\int_{\R^d}\frac{|u_n(x)|^q}{|x|^{sq-\al-q\be/p}}\,dx&=\om_{d-1}\left[\left(n^{\de}-n^{-\de}\right)^q\int_{0}^{\frac{1}{n}}r^{q\de-1}\,dr+\int_{\frac{1}{n}}^{n}\left(r^{-\de}-n^{-\de}\right)^qr^{q\de-1}\,dr\right]\\
\nonumber&=\om_{d-1}\left[\frac{1}{q\de}\left(1-n^{-2\de}\right)^q+\int_{\frac{1}{n^2}}^1\left(1-t^{\de}\right)^q\frac{dt}{t}\right]\\
&=\om_{d-1}\left(2\log n+O(1)\right),\quad n\rightarrow\infty.
\end{align}
Moreover, using polar coordinates, similarly as in the first part of the proof (notice that \eqref{holder} becomes an equality for radial functions),
\begin{align*}
\int_{\R^d}&\left(\int_{\R^d}\frac{|u_n(x)-u_n(y)|^p}{|x-y|^{d+sp}}|y|^{\be}\,dy\right)^{\frac{q}{p}}|x|^{\al}\,dx\\
&=\om_{d-1}\int_{0}^{\infty}r^{q\de-1}\,dr\,\left(\int_{0}^{\infty}\left|u_n(r)-u_n(rt)\right|^pt^{d-1+\be}\Phi(t)\,dt\right)^{\frac{q}{p}}\\
&=:\om_{d-1}\left(K_1+K_2+K_3\right),
\end{align*}
where
\begin{align*}
K_1&=\int_{0}^{\frac{1}{n}}r^{q\de-1}\,dr\left(\int_{0}^{\infty}\left|n^{\de}-n^{-\de}-u_n(rt)\right|^pt^{d-1+\be}\Phi(t)\,dt\right)^{\frac{q}{p}},\\
K_2&=\int_{\frac{1}{n}}^{n}r^{q\de-1}\,dr\left(\int_{0}^{\infty}\left|r^{-\de}-n^{-\de}-u_n(rt)\right|^pt^{d-1+\be}\Phi(t)\,dt\right)^{\frac{q}{p}},\\
K_3&=\int_{n}^{\infty}r^{q\de-1}\,dr\left(\int_{0}^{\infty}\left|u_n(rt)\right|^pt^{d-1+\be}\Phi(t)\,dt\right)^{\frac{q}{p}}.
\end{align*}
We start with the integral $K_1$. Since $\tfrac{q}{p}\leq 1$, it holds $(a+b)^{\frac{q}{p}}\leq a^{\frac{q}{p}}+b^{\frac{q}{p}}$ for $a,b\geq 0$. Hence, we have
\begin{align*}
K_1&\leq\int_{0}^{\frac{1}{n}}r^{q\de-1}\,dr\left(\int_{\frac{1}{nr}}^{\frac{n}{r}}\left(n^{\de}-(rt)^{-\de}\right)^pt^{d-1+\be}\Phi(t)\,dt\right)^{\frac{q}{p}}\\
&\quad+\int_{0}^{\frac{1}{n}}r^{q\de-1}\,dr\left(\int_{\frac{n}{r}}^{\infty}\left(n^{\de}-n^{-\de}\right)^pt^{d-1+\be}\Phi(t)\,dt\right)^{\frac{q}{p}}\\
&\lesssim n^{q\de}\int_{0}^{\frac{1}{n}}r^{q\de-1}\,dr\left(\int_{\frac{1}{nr}}^{\infty}\left(1-t^{-\de}\right)^pt^{d-1+\be}\Phi(t)\,dt\right)^{\frac{q}{p}}\\
&\quad+n^{q\de}\int_{0}^{\frac{1}{n}}r^{q\de-1}\,dr\left(\int_{\frac{n}{r}}^{\infty}t^{d-1+\be}\Phi(t)\,dt\right)^{\frac{q}{p}}\\
&\leq n^{q\de}\int_{0}^{\frac{1}{n}}r^{q\de-1}\,dr\left(\int_{1}^{\infty}\left(1-t^{-\de}\right)^pt^{d-1+\be}\Phi(t)\,dt\right)^{\frac{q}{p}}\\
&\quad +n^{q\de}\int_{0}^{\frac{1}{n}}r^{q\de-1}\,dr\left(\int_{2}^{\infty}t^{d-1+\be}\Phi(t)\,dt\right)^{\frac{q}{p}}=O(1).
\end{align*}
The last conclusion follows from the observation that, since $\Phi(\tfrac{1}{t})=t^{d+sp}\Phi(t)$, it holds $\Phi(t)\approx t^{-d-sp}$, $t\rightarrow\infty$, which gives $\int_{1}^{\infty}\left(1-t^{-\de}\right)^pt^{d-1+\be}\Phi(t)\,dt<\infty$ and $\int_{2}^{\infty}t^{d-1+\be}\Phi(t)\,dt<\infty$, thanks to the assumption $\be<sp$.

In a similar way, we will show that $K_3=O(1)$. We have
\begin{align*}
K_3&\leq\int_{n}^{\infty}r^{q\de-1}\,dr\left(\int_{0}^{\frac{1}{nr}}\left(n^{\de}-n^{-\de}\right)^{p}t^{d-1+\be}\Phi(t)\,dt\right)^{\frac{q}{p}}\\
&\quad+\int_{n}^{\infty}r^{q\de-1}\,dr\left(\int_{\frac{1}{nr}}^{\frac{n}{r}}\left((rt)^{-\de}-n^{-\de}\right)^{p}t^{d-1+\be}\Phi(t)\,dt\right)^{\frac{q}{p}}\\
&=:I_1+I_2.
\end{align*}
Denote $A=\sup_{t\in[0,\frac{1}{2}]}\Phi(t)$. By the assumptions $\be>-d$, we have 
\begin{align*}
I_1&\lesssim A^{\frac{q}{p}}n^{q\de}\int_{n}^{\infty}r^{q\de-1}\,dr\left(\int_{0}^{\frac{1}{nr}}t^{d+\be-1}\,dt\right)^{\frac{q}{p}}\approx n^{q\de-\frac{q}{p}(d+\be)}\int_{n}^{\infty}r^{q\de-\frac{q}{p}(d+\be)-1}\,dr\\
&\approx n^{2\left(q\de-\frac{q}{p}(d+\be)\right)}\longrightarrow 0,\quad\,n\rightarrow\infty,
\end{align*}
since $q\de-\frac{q}{p}(d+\be)=d\left(1-\tfrac{q}{p}\right)+\al-sq<0$. Therefore, $I_1\rightarrow 0$.

For $I_2$, substituting $r\mapsto nr$, we get
\begin{align*}
I_2&=\int_{1}^{\infty}r^{q\de-1}\,dr\left(\int_{\frac{1}{n^2 r}}^{\frac{1}{r}}\left((rt)^{-\de}-1\right)^pt^{d+\be-1}\Phi(t)\,dt\right)^{\frac{q}{p}}\\
&\leq\int_{1}^{\infty}\frac{dr}{r}\left(\int_{\frac{1}{n^2 r}}^{\frac{1}{r}}\left(t^{-\de}-1\right)^pt^{d+\be-1}\Phi(t)\,dt\right)^{\frac{q}{p}}\\
&\leq \int_{1}^{2}\frac{dr}{r}\left(\int_{0}^{1}\left(t^{-\de}-1\right)^pt^{d+\be-1}\Phi(t)\,dt\right)^{\frac{q}{p}}\\
&\quad+\int_{2}^{\infty}\frac{dr}{r}\left(\int_{0}^{\frac{1}{r}}\left(t^{-\de}-1\right)^pt^{d+\be-1}\Phi(t)\,dt\right)^{\frac{q}{p}}.
\end{align*}
The former integral above is clearly finite and for the latter, we have the bound
\begin{align*}
  \int_{2}^{\infty}&\frac{dr}{r}\left(\int_{0}^{\frac{1}{r}}\left(t^{-\de}-1\right)^pt^{d+\be-1}\Phi(t)\,dt\right)^{\frac{q}{p}}\lesssim\int_{2}^{\infty}\frac{dr}{r}\left(\int_{0}^{\frac{1}{r}}t^{d+\be-p\de-1}\,dt\right)^{\frac{q}{p}}\\
  &\lesssim\int_{2}^{\infty}\frac{dr}{r^{1+q(d+\be-p\de)/p}}<\infty,
\end{align*}
since $\frac{q}{p}(d+\be-p\de)=sq-\al-\left(1-\frac{q}{p}\right)d>0$. Overall, we obtain that $K_3=O(1),\,n\rightarrow\infty$. 

That leaves us with the integral $K_2$. We have
\begin{align*}
K_2&\leq \int_{\frac{1}{n}}^nr^{q\de-1}\,dr\left(\int_{0}^{\frac{1}{rn}}\left|r^{-\de}-n^{\de}\right|^pt^{d+\be-1}\Phi(t)\,dt\right)^{\frac{q}{p}}\\
&\quad+\int_{\frac{1}{n}}^nr^{q\de-1}\,dr\left(\int_{\frac{1}{rn}}^{\frac{n}{r}}\left|r^{-\de}-(rt)^{-\de}\right|^pt^{d+\be-1}\Phi(t)\,dt\right)^{\frac{q}{p}}\\
&\quad+\int_{\frac{1}{n}}^nr^{q\de-1}\,dr\left(\int_{{\frac{n}{r}}}^{\infty}\left|r^{-\de}-n^{-\de}\right|^pt^{d+\be-1}\Phi(t)\,dt\right)^{\frac{q}{p}}\\
&=:J_1+J_2+J_3.
\end{align*}
We want to show that $J_1+J_3=O(1),\,n\rightarrow\infty$. To this end, recall that, by \cite[(3.665)]{MR2360010},
\begin{equation}\label{2F1}
\Phi(t)=\om_{d-2}\text{B}\left(\tfrac{d-1}{2},\tfrac{1}{2}\right)\2F1\left(\tfrac{d+sp}{2},\tfrac{2+sp}{2},\tfrac{d}{2};t^2\right),
\end{equation}
where $\text{B}$ is the Beta function and $\2F1$ is the hypergeometric function. Hence, $\Phi$ is increasing on $[0,1)$. In consequence,
\begin{align*}
J_1&\leq\int_{\frac{1}{n}}^{n}r^{q\de-1}\left(n^{\de}-r^{-\de}\right)^{q}\Phi\left(\tfrac{1}{rn}\right)^{\frac{q}{p}}\left(\int_{0}^{\frac{1}{rn}}t^{d+\be-1}\,dt\right)^{\frac{q}{p}}\\
&\lesssim\int_{\frac{1}{n}}^{n}r^{q\de-1}\left(n^{\de}-r^{-\de}\right)^{q}\Phi\left(\tfrac{1}{rn}\right)^{\frac{q}{p}}(rn)^{-(d+\be)\frac{q}{p}}\,dr\\
&=\int_{\frac{1}{n^2}}^{1}u^{-q\de-1}\left(1-u^{\de}\right)^q\Phi(u)^{\frac{q}{p}}u^{(d+\be)\frac{q}{p}}\,du\\
&\lesssim \int_{\frac{1}{n^2}}^{\frac{1}{2}}u^{-q\de+\frac{q}{p}(d+\be)-1}\,du+\int_{\frac{1}{2}}^{1}\left(1-u^{\de}\right)^q\Phi(u)^{\frac{q}{p}}\,du.
\end{align*}
The first integral is finite for $n\rightarrow\infty$. For the second, we observe that by \eqref{2F1} and properties of the hypergeometric function, we have $\Phi(u)\approx (1-u)^{-1-sp}$, for $u\rightarrow 1^{-}$, as it was shown in \cite[beetween (3.4) and (3.5)]{MR2469027}. Since $1-u^{\de}\approx 1-u$, we have
$$
(1-u^{\de})^q\Phi(u)^{\frac{q}{p}}\approx \frac{1}{(1-u)^{\frac{q}{p}-q(1-s)}}
$$
and the latter is integrable for $u\rightarrow 1^-$, since $\frac{q}{p}-q(1-s)\leq1-q(1-s)<1$. In consequence, we obtain that $J_1=O(1)$. 

Dealing with the integral $J_3$, by substitution $r\mapsto rn$ and then $t\mapsto\frac{1}{t}$, we get
\begin{align*}
J_3&=\int_{\frac{1}{n^2}}^1r^{q\de-1}\left(r^{-\de}-1\right)^q\,dr\left(\int_{\frac{1}{r}}^{\infty}t^{d+\be-1}\Phi(t)\,dt\right)^{\frac{q}{p}}\\
&=\int_{\frac{1}{n^2}}^1r^{q\de-1}\left(r^{-\de}-1\right)^q\,dr\left(\int_{0}^{r}t^{-d-\be-1}\Phi\left(\tfrac{1}{t}\right)\,dt\right)^{\frac{q}{p}}\\
&=\int_{\frac{1}{n^2}}^1r^{q\de-1}\left(r^{-\de}-1\right)^q\,dr\left(\int_{0}^{r}t^{sp-\be-1}\Phi\left(t\right)\,dt\right)^{\frac{q}{p}}\\
&\leq\int_{\frac{1}{n^2}}^1r^{q\de-1}\left(r^{-\de}-1\right)^q\Phi\left(r\right)^{\frac{q}{p}}\,dr\left(\int_{0}^{r}t^{sp-\be-1}\,dt\right)^{\frac{q}{p}}\\
&\lesssim\int_{\frac{1}{n^2}}^1r^{q\de+\frac{q}{p}(sp-\be)-1}\left(r^{-\de}-1\right)^q\Phi\left(r\right)^{\frac{q}{p}}\,dr\\
&\lesssim\int_{\frac{1}{n^2}}^{\frac{1}{2}}r^{\frac{q}{p}(sp-\be)-1}\,dr+\int_{\frac{1}{2}}^{1}\left(1-r^{\de}\right)^q\Phi\left(r\right)^{\frac{q}{p}}\,dr=O(1),
\end{align*}
where we used the assumption $\be<sp$.

Now, the last integral left to evaluate is $J_2$. But this case is the simplest, as we have,
\begin{align*}
J_2&=\int_{\frac{1}{n}}^{n}\frac{dr}{r} \left(\int_{\frac{1}{rn}}^{\frac{n}{r}}\left|1-t^{-\de}\right|^pt^{d+\be-1}\Phi(t)\,dt\right)^{\frac{q}{p}}\\
&\leq\int_{\frac{1}{n}}^{n}\frac{dr}{r} \left(\int_{0}^{\infty}\left|1-t^{-\de}\right|^pt^{d+\be-1}\Phi(t)\,dt\right)^{\frac{q}{p}}\\
&=2\,\mathcal{C}^{p,q}_{d,s,\al,\be}\log n.
\end{align*}
Hence, summarizing all the computations, we showed that 
$$
K_1+K_2+K_3=2\,\mathcal{C}^{p,q}_{d,s,\al,\be}\log n+O(1),\quad n\rightarrow\infty.
$$ 
Thus, recalling \eqref{log},
\begin{align*}
\frac{\displaystyle\int_{\R^d}\left(\int_{\R^d}\frac{|u_n(x)-u_n(y)|^p}{|x-y|^{d+sp}}|y|^{\be}\,dy\right)^{\frac{q}{p}}|x|^{\al}\,dx}{\displaystyle\int_{\R^d}\frac{|u_n(x)|^q}{|x|^{sq-\al-q\be/p}}\,dx}=\frac{2\,\mathcal{C}^{p,q}_{d,s,\al,\be}\log n+O(1)}{2\log n+O(1)}\longrightarrow \mathcal{C}^{p,q}_{d,s,\al,\be},
\end{align*}
as $n$ tends to infinity. Consequently, the constant $\mathcal{C}^{p,q}_{d,s,\al,\be}$ is indeed sharp in the inequality \eqref{HardyTriebelLizorkin}, for the case $d+\al+\frac{q\be}{p}s-sq>0$.

We now turn to the case, when $\de=d+\al+\frac{q\be}{p}-sq<0$. Recall that the inversion $x\mapsto T(x)=\tfrac{x}{|x|^2}$ has Jacobian equal to $(-1)^d|x|^{-2d}$ and satisfies $|T(x)-T(y)|=\tfrac{|x-y|}{|x||y|}$. Thus, substituting $x\mapsto T(x)$, $y\mapsto T(y)$, we have
\begin{align*}
\int_{\R^d}&\left(\int_{\R^d}\frac{|u(x)-u(y)|^p}{|x-y|^{d+sp}}|y|^{\be}\,dy\right)^{\frac{q}{p}}|x|^{\al}\,dx\\
&=\int_{\R^d}\left(\int_{\R^d}\frac{|u(T(x))-u(T(y))|^p}{|x-y|^{d+sp}}|y|^{\be'}\,dy\right)^{\frac{q}{p}}|x|^{\al'}\,dx
\end{align*}
and
$$
\int_{\R^d}\frac{|u(x)|^q}{|x|^{sq-\al-\frac{q\be}{p}}}\,dx=\int_{\R^d}\frac{|u(T(x))|^q}{|x|^{sq-\al'-\frac{q\be'}{p}}}\,dx
$$
where $\al'=\tfrac{qd}{p}+sq-2d-\al$ and $\be'=sp-d-\be$. We now observe that $\al'$, $\be'$ satisfy \eqref{assumption1},
$$
\de'=d+\al'+\frac{q\be'}{p}-sq=-d-\al-\frac{q\be}{p}+sq=-\de>0,
$$
and $\mathcal{C}^{p,q}_{d,s,\al,\be}=\mathcal{C}^{p,q}_{d,s,\al',\be'}$. Thus, we may reduce the case of $\de<0$ to the previous situation. The proof is now complete.
\end{proof}
\begin{proof}[Proof of Theorem \ref{thm2}]
  We will present the proof for $d=1$ separately, as it turns out to be significantly simpler. Therefore, let us treat the case $d=1$. Substituting $y\mapsto xy$ in the inner integral, we have
\begin{align*}
\int_{0}^{\infty}\left(\int_{0}^{\infty}\frac{|u(x)-u(y)|^p}{|x-y|^{1+sp}}y^{\be}\,dy\right)^{\frac{q}{p}}x^\al\,dx&=\int_{0}^{\infty}x^{\ga-1}\left(\int_{0}^{\infty}\frac{|u(x)-u(xy)|^p}{|1-y|^{1+sp}}y^{\be}\,dy\right)^{\frac{q}{p}}dx,
\end{align*}  
where $\ga=1+\al+q\be/p-sq$. By the Minkowski's integral inequality,
\begin{align*}
    \int_{0}^{\infty}x^{\ga-1}&\left(\int_{0}^{\infty}\frac{|u(x)-u(xy)|^p}{|1-y|^{1+sp}}y^{\be}\,dy\right)^{\frac{q}{p}}dx\\
    &\geq\left(\int_{0}^{\infty}\frac{y^{\be}}{|1-y|^{1+sp}}\left(\int_{0}^{\infty}\left|u(x)-u(xy)\right|^qx^{\ga-1}\,dx\right)^{\frac{p}{q}}dy\right)^{\frac{q}{p}}.
\end{align*}
Applying Minkowski's inequality, similarly as in the proof of Theorem \ref{thm1} we can show that, 
\begin{align*}
 \int_{0}^{\infty}\left|u(x)-u(xy)\right|^qx^{\ga-1}\,dx\geq\left|1-y^{-\ga/q}\right|^{q}\int_{0}^{\infty}|u(x)|^qx^{\ga-1}\,dx.   
\end{align*}
Consequently, \eqref{half-space} for $d=1$ holds true.

Let now $d\geq 2$. We adapt the convention $x=(x',x_d)$, with $x'\in\R^{d-1}$, $x_d>0$ and we write $u(x)=u(x',x_d)$. We than have
\begin{align}\label{U}
&\int_{\R^d_+}\left(\int_{\R^d_+}\frac{|u(x)-u(y)|^p}{|x-y|^{d+sp}}y_d^{\be}\,dy\right)^{\frac{q}{p}}x_d^{\al}\,dx\\
\nonumber&\quad=\int_{0}^{\infty}\int_{\R^{d-1}}x_d^{\al}\,dx'\,dx_d\left(\int_{0}^{\infty}\int_{\R^{d-1}}\frac{|u(x',x_d)-u(y',y_d)|^p}{\left(|x'-y'|^2+(x_d-y_d)^2\right)^{\frac{d+sp}{2}}}y_d^{\be}\,dy'\,dy_d\right)^{\frac{q}{p}}\\
\nonumber&\quad=\int_{0}^{\infty}\int_{\R^{d-1}}x_d^{\al+\frac{q(\be+1)}{p}}\,dx'\,dx_d\left(\int_{0}^{\infty}\int_{\R^{d-1}}\frac{|u(x',x_d)-u(y',x_dy_d)|^p}{\left(|x'-y'|^2+x_d^2(1-y_d)^2\right)^{\frac{d+sp}{2}}}y_d^{\be}\,dy'\,dy_d\right)^{\frac{q}{p}},
\end{align}
where the last equality follows from substituting $y_d\mapsto x_dy_d$ in the inner integral. By Hölder's inequality and \cite[(6.2.1)]{MR1225604}, we have
\begin{align*}
 &\int_{\R^{d-1}}\frac{|u(x',x_d)-u(y',x_dy_d)|^p}{\left(|x'-y'|^2+x_d^2(1-y_d)^2\right)^{\frac{d+sp}{2}}}\,dy'\\
 &\geq\left(\int_{\R^{d-1}}\frac{|u(x',x_d)-u(y',x_dy_d)|^q}{\left(|x'-y'|^2+x_d^2(1-y_d)^2\right)^{\frac{d+sp}{2}}}\,dy'\right)^{\frac{p}{q}}\left(\int_{\R^{d-1}}\frac{dy'}{\left(|x'-y'|^2+x_d^2(1-y_d)^2\right)^{\frac{d+sp}{2}}}\right)^{1-\frac{p}{q}}\\
 &=\mathcal{D}^{1-\frac{p}{q}}\left(x_d\left|1-y_d\right|\right)^{(1+sp)(\frac{p}{q}-1)}\left(\int_{\R^{d-1}}\frac{|u(x',x_d)-u(y',x_dy_d)|^q}{\left(|x'-y'|^2+x_d^2(1-y_d)^2\right)^{\frac{d+sp}{2}}}\,dy'\right)^{\frac{p}{q}},\\
\end{align*}
where the constant $\mathcal{D}$ is given by
\begin{equation}\label{C}
\mathcal{D}=\mathcal{D}(d,s,p)=\frac{\pi^{\frac{d-1}{2}}\Gamma\left(\frac{1+sp}{2}\right)}{\Gamma\left(\frac{d+sp}{2}\right)}.
\end{equation}
Therefore, applying Minkowski's integral inequality we obtain that the right-hand side of \eqref{U} is bigger or equal to
\begin{align*}
 &\mathcal{D}^{\frac{q}{p}-1}\int_{0}^{\infty}\int_{\R^{d-1}}x_d^{\al+\frac{q\be}{p}+1+(p-q)s}\,dx'\,dx_d\\
 &\quad\times\left(\int_{0}^{\infty}y_d^{\be}\left|1-y_d\right|^{(1+sp)(\frac{p}{q}-1)}\left(\int_{\R^{d-1}}\frac{|u(x',x_d)-u(y',x_dy_d)|^q}{\left(|x'-y'|^2+x_d^2(1-y_d)^2\right)^{\frac{d+sp}{2}}}\,dy'\right)^{\frac{p}{q}}dy_d\right)^{\frac{q}{p}}\\
 &\geq \mathcal{D}^{\frac{q}{p}-1}\Bigg[\int_{0}^{\infty}y_d^{\be}\left|1-y_d\right|^{(1+sp)(\frac{p}{q}-1)}\,dy_d\\
 &\quad\times\left(\int_{0}^{\infty}x_d^{\al+\frac{q\be}{p}+1+(p-q)s}\,dx_d\int_{\R^{d-1}}\int_{\R^{d-1}}\frac{|u(x',x_d)-u(y',x_dy_d)|^q}{\left(|x'-y'|^2+x_d^2(1-y_d)^2\right)^{\frac{d+sp}{2}}}\,dy'\,dx'\right)^{\frac{p}{q}}\Bigg]^{\frac{q}{p}}.
\end{align*}
Substituting $x'\mapsto x_dx'$, $y'\mapsto x_dy'$ we have
\begin{align*}
 &\int_{0}^{\infty}x_d^{\al+\frac{2\be}{p}+1+(p-q)s}\,dx_d\int_{\R^{d-1}}\int_{\R^{d-1}}\frac{\left|u(x',x_d)-u(y',x_dy_d)\right|^q}{\left(|x'-y'|^2+x_d^2(1-y_d)^2\right)^{\frac{d+sp}{2}}}\,dy'\,dx'\\
 &=\int_{0}^{\infty}\int_{\R^{d-1}}\int_{\R^{d-1}}x_d^{\ga}\frac{\left|u(x'x_d,x_d)-u(y'x_d,x_dy_d)\right|^q}{\left(|x'-y'|^2+(1-y_d)^2\right)^{\frac{d+sp}{2}}}\,dy'\,dx'\,dx_d,
\end{align*}
with $\ga=\al+\tfrac{q\be}{p}+d-1-sq$. 

Moreover, observe that
\begin{align*}
\int_{\R^{d-1}}&\int_{\R^{d-1}}\int_{0}^{\infty}\frac{\left|u(x'x_d,x_d)\right|^qx_d^{\ga}}{\left(\left|x'-y'\right|^2+(1-y_d)^2\right)^{\frac{d+sp}{2}}}\,dx_d\,dy'\,dx'\\
&=\int_{\R^{d-1}}\int_{0}^{\infty}\left|u(x'x_d,x_d)\right|^qx_d^{\ga}\,dx_d\,dx'\int_{\R^{d-1}}\frac{dy'}{\left(\left|x'-y'\right|^2+(1-y_d)^2\right)^{\frac{d+sp}{2}}}\\
&=\mathcal{D}\left|1-y_d\right|^{-1-sp}\int_{\R^{d-1}}\int_{0}^{\infty}\left|u(x'x_d,x_d)\right|^qx_d^{\ga}\,dx_d\,dx'\\
&=\mathcal{D}\left|1-y_d\right|^{-1-sp}\int_{\R^{d-1}}\int_{0}^{\infty}\left|u(x',x_d)\right|^qx_d^{\ga-d+1}\,dx_d\,dx'\\
&=\mathcal{D}\left|1-y_d\right|^{-1-sp}\int_{\R^{d}_+}\frac{|u(x)|^q}{x_d^{sq-\al-\frac{q\be}{p}}}\,dx,
\end{align*}
with $\mathcal{D}$ as in \eqref{C}. In the same manner,
\begin{align*}
\int_{\R^{d-1}}&\int_{\R^{d-1}}\int_{0}^{\infty}\frac{\left|u(y'x_d,x_dy_d)\right|^qx_d^{\ga}}{\left(\left|x'-y'\right|^2+(1-y_d)^2\right)^{\frac{d+sp}{2}}}\,dx_d\,dy'\,dx'\\
&=y_{d}^{-(\ga+1)}\int_{\R^{d-1}}\int_{\R^{d-1}}\int_{0}^{\infty}\frac{\left|u\left(\tfrac{y'x_d}{y_d},x_d\right)\right|^qx_d^{\ga}}{\left(\left|x'-y'\right|^2+(1-y_d)^2\right)^{\frac{d+sp}{2}}}\,dx_d\,dy'\,dx'\\
&=\mathcal{D}y_{d}^{-(\ga+1)}\left|1-y_d\right|^{-1-sp}\int_{\R^{d-1}}\int_{0}^{\infty}\left|u\left(\tfrac{y'x_d}{y_d},x_d\right)\right|^qx_d^{\ga}\,dx_d\,dy'\\
&=\mathcal{D}y_d^{d-2-\ga}\left|1-y_d\right|^{-1-sp}\int_{\R^{d}_+}\frac{\left|u(x)\right|^q}{x_d^{sq-\al-\frac{q\be}{p}}}\,dx.
\end{align*}
We now apply Minkowski's inequality to get
\begin{align*}
&\left(\int_{\R^{d-1}}\int_{\R^{d-1}}\int_{0}^{\infty}\frac{\left|u(x'x_d,x_d)\right|^qx_d^{\ga}}{\left(\left|x'-y'\right|^2+(1-y_d)^2\right)^{\frac{d+sp}{2}}}\,dx_d\,dy'\,dx'\right)^{1/q}\\
&\quad\leq\left(\int_{\R^{d-1}}\int_{\R^{d-1}}\int_{0}^{\infty}\frac{\left|u(x'x_d,x_d-u(y'x_d,x_dy_d)\right|^qx_d^{\ga}}{\left(\left|x'-y'\right|^2+(1-y_d)^2\right)^{\frac{d+sp}{2}}}\,dx_d\,dy'\,dx'\right)^{1/q}\\
&\quad+\left(\int_{\R^{d-1}}\int_{\R^{d-1}}\int_{0}^{\infty}\frac{\left|u(y'x_d,x_dy_d)\right|^qx_d^{\ga}}{\left(\left|x'-y'\right|^2+(1-y_d)^2\right)^{\frac{d+sp}{2}}}\,dx_d\,dy'\,dx'\right)^{1/q}\\
&\quad=\left(\int_{\R^{d-1}}\int_{\R^{d-1}}\int_{0}^{\infty}\frac{\left|u(x'x_d,x_d-u(y'x_d,x_dy_d)\right|^qx_d^{\ga}}{\left(\left|x'-y'\right|^2+(1-y_d)^2\right)^{\frac{d+sp}{2}}}\,dx_d\,dy'\,dx'\right)^{1/q}\\
&\quad+\left(\mathcal{D}|1-y_d|^{-1-sp}y_d^{d-2-\ga}\int_{\R^{d}_+}\frac{|u(x)|^q}{x_d^{sq-\al-\frac{q\be}{p}}}\,dx\right)^{1/q}.
\end{align*}
Repeating similar calculations as above for 
$$
\left(\int_{\R^{d-1}}\int_{\R^{d-1}}\int_{0}^{\infty}\frac{\left|u(y'x_d,x_dy_d)\right|^qx_d^{\ga}}{\left(\left|x'-y'\right|^2+(1-y_d)^2\right)^{\frac{d+sp}{2}}}\,dx_d\,dy'\,dx'\right)^{1/q}
$$
on the left-hand side, we can show that the obtained estimates yield
\begin{align*}
&\int_{0}^{\infty}\int_{\R^{d-1}}\int_{\R^{d-1}}x_d^{\ga}\frac{\left|u(x'x_d,x_d)-u(y'x_d,x_dy_d)\right|^q}{\left(|x'-y'|^2+(1-y_d)^2\right)^{\frac{d+sp}{2}}}\,dy'\,dx'\,dx_d\\
&\quad\geq \mathcal{D}\left|1-y_d\right|^{-1-sp}\left|1-y_d^{(d-2-\ga)/q}\right|^q\int_{\R^{d}_+}\frac{|u(x)|^q}{x_d^{sq-\al-\frac{q\be}{p}}}\,dx.
\end{align*}
Thus, easy rearranging ends the proof.

The last part of the proof is to justify optimality of the constant \eqref{DTriebelLizorkin}. To this end, as usually in the setting of Hardy-type inequalities on half-spaces, we will reduce the multidimensional case to the one-dimensional problem. More precisely, for $d=1$, sharpness of the constant $\mathcal{D}^{p,q}_{1,s,\al,\be}$ can be shown exactly like in the proof of Theorem \ref{thm1}, with minor modifications, taking the same test functions $u_n$ as in \eqref{u_n}. For $d\geq 2$, we let $\varphi\in C_c^{1}(\R^{d-1})$ and, for $N>0$, we define
$$
\varphi_N(x')=\frac{N^{\frac{1-d}{q}}}{\|\varphi\|_{L^q(\R^{d-1})}}\varphi\left(\frac{x'}{N}\right).
$$
For $x\in\R^d_+$ and $\eta\in C_c([0,\infty))$ if $sq-\al-\tfrac{q\be}{p}<1$ or $\eta\in C_c((0,\infty))$ if $sq-\al-\tfrac{q\be}{p}>1$, let now $\psi_N(x)=\varphi_N(x')\eta(x_d)$. Obviously, 
$$
\int_{\R^{d}_+}\frac{|\psi_N(x)|^q}{x_d^{sq-\al-q\be/p}}\,dx=\int_{0}^{\infty}\frac{|\eta(x_d)|^q}{x_d^{sq-\al-q\be/p}}\,dx_d.
$$
Moreover, since  
$$
\left|\psi_N(x)-\psi_N(y)\right|\leq\left|\varphi_N(x')\right|\left|\eta(x_d)-\eta(y_d)\right|+\left|\eta(y_d)\right|\left|\varphi_N(x')-\varphi_N(y')\right|,
$$
by Minkowski's inequality (applied twice) it holds
\begin{align*}
&\left(\int_{\R^d_+}\left(\int_{\R^d_+}\frac{|\psi_N(x)-\psi_N(y)|^p}{|x-y|^{d+sp}}y_d^{\be}\,dy\right)^{\frac{q}{p}}x_d^{\al}\,dx\right)^{\frac{1}{q}}\\
&\quad\quad\leq\left(\int_{\R^d_+}\left(\int_{\R^d_+}\frac{\left|\varphi_N(x')\right|^p\left|\eta(x_d)-\eta(y_d)\right|^p}{|x-y|^{d+sp}}y_d^{\be}\,dy\right)^{\frac{q}{p}}x_d^{\al}\,dx\right)^{\frac{1}{q}}\\
&\quad\quad+\left(\int_{\R^d_+}\left(\int_{\R^d_+}\frac{\left|\eta(y_d)\right|^p\left|\varphi_N(x')-\varphi_N(y')\right|^p}{|x-y|^{d+sp}}y_d^{\be}\,dy\right)^{\frac{q}{p}}x_d^{\al}\,dx\right)^{\frac{1}{q}}\\
&\quad\quad=:I_1+I_2.
\end{align*}
We have
\begin{equation}\label{I1}
I_1=\mathcal{D}^{\frac{1}{p}}\left(\int_{0}^{\infty}\left(\int_{0}^{\infty}\frac{\left|\eta(x_d)-\eta(y_d)\right|^p}{\left|x_d-y_d\right|^{1+sp}}y_d^{\be}\,dy_d\right)^{\frac{q}{p}}x_d^{\al}\,dx_d\right)^{\frac{1}{q}}. 
\end{equation}
Next, we will show that $I_2\rightarrow 0$, as $N\rightarrow\infty$. To this end, we estimate
\begin{align*}
I_2&\leq \left(\int_{\R^d_+}\left(\int_{\R^d_+}\frac{\left|\eta(y_d)-\eta(x_d)\right|^p\left|\varphi_N(x')-\varphi_N(y')\right|^p}{|x-y|^{d+sp}}y_d^{\be}\,dy\right)^{\frac{q}{p}}x_d^{\al}\,dx\right)^{\frac{1}{q}}\\
&+\left(\int_{\R^d_+}\left(\int_{\R^d_+}\frac{\left|\eta(x_d)\right|^p\left|\varphi_N(x')-\varphi_N(y')\right|^p}{|x-y|^{d+sp}}y_d^{\be}\,dy\right)^{\frac{q}{p}}x_d^{\al}\,dx\right)^{\frac{1}{q}}
\\&=:J_1+J_2.
\end{align*}
Observe that, by Young's inequality, for sufficiently small $\varepsilon>0$,
$$
|x-y|^{d+sp}\gtrsim|x'-y'|^{d-1+\varepsilon}|x_d-y_d|^{1+sp-\varepsilon}. 
$$
Consequently,
\begin{align*}
J_1^q&\lesssim  \int_{\R^d_+}\left(\int_{\R^d_+}\frac{\left|\eta(y_d)-\eta(x_d)\right|^p\left|\varphi_N(x')-\varphi_N(y')\right|^p}{|x'-y'|^{d-1+\varepsilon}|x_d-y_d|^{1+sp-\varepsilon}}y_d^{\be}\,dy\right)^{\frac{q}{p}}x_d^{\al}\,dx\\
&=\int_{\R^{d-1}}\left(\int_{\R^{d-1}}\frac{\left|\varphi_N(x')-\varphi_N(y')\right|^p}{|x'-y'|^{d-1+\varepsilon}}\,dy'\right)^{\frac{q}{p}}dx'\\
&\quad\times\int_{0}^{\infty}\left(\int_{0}^{\infty}\frac{|\eta(x_d)-\eta(y_d)|^p}{|x_d-y_d|^{1+sp-\varepsilon}}y_d^{\be}\,dy_d\right)^{\frac{q}{p}}x_d^{\al}\,dx_d.
\end{align*}
Clearly, it holds
\begin{align*}
 \int_{\R^{d-1}}\left(\int_{\R^{d-1}}\frac{\left|\varphi_N(x')-\varphi_N(y')\right|^p}{|x'-y'|^{d-1+\varepsilon}}\,dy'\right)^{\frac{q}{p}}dx'=\frac{1}{N^{q\varepsilon/p}}\int_{\R^{d-1}}\left(\int_{\R^{d-1}}\frac{\left|\varphi(x')-\varphi(y')\right|^p}{|x'-y'|^{d-1+\varepsilon}}\,dy'\right)^{\frac{q}{p}}dx'.
\end{align*}
If $\varepsilon$ is sufficiently small, by Lemma \ref{Lemma}, both Triebel--Lizorkin seminorms above are finite, hence, due to the factor $N^{-q\varepsilon/p}$ in front, the integral $J_1$ tends to zero, as $N\rightarrow\infty$.

Finally, we have 
$$
J_2\leq\left(\int_{\R^d}\left(\int_{\R^d}\frac{\left|\varphi_N(x')-\varphi_N(y')\right|^p}{|x-y|^{d+sp}}|y_d|^{\be}\,dy\right)^{\frac{q}{p}}\left|\eta(|x_d|)\right|^q|x_d|^{\al}\,dx\right)^{\frac{1}{q}}
$$
Now, using exactly the same arguments as in \cite[Proof of Theorem 1.1]{Vivek}, it may be shown that $J_2\rightarrow 0$, as $N\rightarrow\infty$. The proof uses, among others, the fact that the function $w(y_d)=|y_d|^{\be}$ is a Muckenhoupt $A_1$ weight for $\be\in(-1,0]$, we refer again to\cite[Theorem 1.1]{MR3900847}. We omit the details.

Overall, since 
$$
\frac{\mathcal{D}^{p,q}_{d,s,\al,\be}}{\mathcal{D}^{p,q}_{1,s,\al,\be}}=\mathcal{D}^{\frac{q}{p}},
$$
sharpness of the constant $\mathcal{D}^{p,q}_{d,s,\al,\be}$ for $d\geq 2$ follows now from \eqref{I1} and sharpness of the constant $\mathcal{D}^{p,q}_{1,s,\al,\be}$. Hence, the proof is complete.
\end{proof}

\end{document}